\documentclass[oneside,english,a4paper]{article}
\usepackage{graphicx} 
\usepackage[T1]{fontenc}
\usepackage{amsmath}
\usepackage{amsthm}
\usepackage{amssymb}
\usepackage[all]{xy}
\usepackage{amsbsy}
\usepackage{tensor}

\usepackage{comment}
\usepackage{color}
\usepackage[nocompress]{cite}
\usepackage{tikz-cd}
\usepackage{mathtools}

\usepackage[margin=23mm]{geometry}
\usepackage{enumitem}
\usepackage[square,sort,comma,numbers]{natbib}
\makeatletter
\newcommand{\nc}{\newcommand}
\theoremstyle{plain} 
\newtheorem{thm}{Theorem}
\newtheorem*{thm*}{Theorem}
\nc{\bthm}{\begin{thm}} \nc{\ethm}{\end{thm}}
\newtheorem{prop}[thm]{Proposition}
\nc{\bprp}{\begin{prop}} \nc{\eprp}{\end{prop}}

\newtheorem{fact}[thm]{Fact}
\nc{\bfct}{\begin{fact}} \nc{\efct}{\end{fact}}
\newtheorem{prob}[thm]{Problem}
\nc{\bprb}{\begin{prob}} \nc{\eprb}{\end{prob}}
\newtheorem{lem}[thm]{Lemma}
\nc{\blem}{\begin{lem}} \nc{\elem}{\end{lem}}
\newtheorem{claim}[thm]{Claim}
\nc{\bclm}{\begin{claim}} \nc{\eclm}{\end{claim}}
\newtheorem{cor}[thm]{Corollary}
\nc{\bcor}{\begin{cor}} \nc{\ecor}{\end{cor}}
\newtheorem{conj}[thm]{Conjecture}
\nc{\bcnj}{\begin{conj}} \nc{\ecnj}{\end{conj}}
\theoremstyle{definition}
\newtheorem{defn}[thm]{Definition}
\nc{\bdfn}{\begin{defn}} \nc{\edfn}{\end{defn}}
\newtheorem{observation}[thm]{Observation}
\nc{\bobs}{\begin{observation}} \nc{\eobs}{\end{observation}}
\theoremstyle{remark}
\newtheorem{rem}[thm]{Remark}
\nc{\brem}{\begin{rem}} \nc{\erem}{\end{rem}}
\newtheorem{cnv}[thm]{Convention}
\nc{\bcnv}{\begin{cnv}} \nc{\ecnv}{\end{cnv}}
\newtheorem{exam}[thm]{Example}
\nc{\bexm}{\begin{exam}} \nc{\eexm}{\end{exam}}

\newtheorem{question}[thm]{Question}
\nc{\bpf}{\begin{proof}} \nc{\epf}{\end{proof}}
\nc{\be}{\begin{enumerate}}
	\nc{\ee}{\end{enumerate}}
\nc{\bi}{\begin{itemize}}
	\nc{\itm}{\item}
	\nc{\ei}{\end{itemize}}
\nc{\invlim}{\lim_{\leftarrow}}
\nc{\dirlim}{\lim_{\rightarrow}}
\nc{\mm}{\mathbf{m}}
\nc{\nn}{\mathbf{n}}
\nc{\kk}{\mathbf{k}}
\nc{\FF}{\mathcal{F}}
\nc{\CC}{\mathcal{C}}
\nc{\Span}{\operatorname{span}}
\nc{\Img}{\operatorname{Im}}
\nc{\rank}{\operatorname{rank}}
\nc{\proj}{\operatorname{proj}}
\nc{\F}{\mathbb{F}}
\nc{\Z}{\mathbb{Z}}
\nc{\Q}{\mathbb{Q}}
\nc{\Br}{\operatorname{Br}}
\nc{\res}{\operatorname{res}}
\nc{\sep}{\operatorname{sep}}
\nc{\Gal}{\operatorname{Gal}}

\title{Free product of Demushkin groups as absolute Galois group}
\author{Tamar Bar-On}
\date{}
\begin{document}
	\maketitle
	\begin{abstract}
	We prove that a free profinite (pro-$p$) product over a set converging to 1 of countably many  Demushkin groups of rank $\aleph_0$, $G_i$, that can be realized as absolute Galois groups, is isomorphic to an absolute Galois group if and only if $\log_pq(G_i)\to \infty$. 
	\end{abstract}
	
	\section*{Introduction}
	
	The famous Elementary Type Conjecture by Ido Efrat (\cite{efrat1995orderings}) suggests that the class of finitely generated pro-$p$ groups which can be realized as maximal pro-$p$ Galois groups of fields containing a primitive $p$-th root of unity is the minimal class of groups that can be constructed from finitely generated free pro-$p$ groups and certain Demushkin groups- those which  occur as maximal pro-$p$ groups over local fields containing a primitive $p$-th root of unity- by forming free pro-$p$ products and certain semidirect products with $\Z_p$. In particular, this class is closed under free pro-$p$ products of finitely many groups.
	
	However, for nonfinitely generated maximal pro-$p$ Galois groups, the situation is much more mysterious. As the building blocks in the finitely generated case are finitely generated free pro-$p$ groups and certain Demushkin groups, in order to start learning maximal pro-$p$ Galois groups of infinite rank it is natural to start with the generalization of these groups to higher rank. While free pro-$p$ groups are naturally defined for every rank, and in fact for every cardinal $\mm$ the free pro-$p$ group of rank $\mm$ can be realized as an absolute Galois group (a full proof can be found, for example, in \cite[Proposition 49]{bar2024demushkin}), the theory of higher rank Demushkin groups is much more complicated. In 1966 Labute presented and classified Demushkin groups of rank $\aleph_0$. In particular he proved that the $p$-Sylow subgroups of finite extension of $\Q_p$ which contain a primitive $p$-th root of unity are pro-$p$ Demushkin groups of rank $\aleph_0$ (\cite[Theorem 5]{labute1966demuvskin}). This work was completed by Min{\'a}{\v{c}} and Ware in their papers \cite{minavc1991demuvskin,minavc1992pro} where they proved that a pro-$p$  Demushkin group $G$ with $q(G)\ne 2$ (to be defined later)  can be realized as a maximal pro-$p$ Galois group if and only if it can be realized as an absolute Galois group, if and only if $s(G)=0$. Here $s(G)$ is an invariant classifies the size of the dualizing module, as presented in \cite{labute1966demuvskin}. A similar result was proved for $q(G)=2$ considering extra invariants.   
	
	Demushkin groups of uncountable rank were only presented in 2024, in the paper \cite{bar2024demushkin}. In that paper it was proved that for every cardinal $\mu$ there exists a pro-$p$ Demushkin group of rank $\mu$ which can be realized as an absolute Galois group. It is still an open question, however, whether every pro-$p$ Demushkin group with the "right" set of invariants can be realized as an absolute Galois group- or even just as a maximal pro-$p$ Galois group. 
	
	As we wish to learn maximal pro-$p$ Galois groups, as well as absolute Galois groups, of infinite rank, the first and most simple operator we shall consider is the \textit{free profinite  product}, and its $p$-version, the free pro-$p$ product, over a set $I$converging to 1 of an arbitrary, possibly infinite, cardinality. In particular we are interested in free profinite (pro-$p$) products of Demushkin groups.
	
	Let $\mathcal{C}$ be a variety of finite groups. By a free pro-$\mathcal{C}$ product over a set converging to 1 we refer to the following definition:
	
	\begin{defn}\cite[Definition 4.1.1]{neukirch2013cohomology}\label{free pro-C product}
		Let $\{G_i\}_{i\in I}$ be a set of profinite groups. A set of continuous homomorphisms $\{\varphi_i:G_i\to G\}_{i\in I}$ where $G$ is a profinite group, is called \textit{converging to 1} if for every open subgroup $U\leq G$, $\varphi_i(G_i)\subseteq U$ for almost all $i\in I$. We say that $G$ is a free profinite product of $\{G_i\}_{i\in I}$ if there is a converging to 1 set of homomoprhisms  $\{\varphi_i:G_i\to G\}$ such that $\overline{\langle \bigcup _{i\in I}\varphi_i(G_i)\rangle} =G$ and for every converging to 1 set of homomorphisms  $\{\psi_i:G_i\to H\}_{i\in I}$ there is a homomorphism $\psi:G\to H$ such that for every $i\in I$, $\psi_i=\psi\circ \varphi_i$. 
	\end{defn}  
\begin{rem}
	One easily deduce that the induced homomorphism $\psi:G\to H$ is unique.
\end{rem}  
	
	free pro-$\mathcal{C}$ product over a set converging to 1 is in fact a special case of the more general notation of a \textit{free pro-$\mathcal{C}$ product over a sheaf}, which is studied extensively in \cite[Chapter 5]{ribes2017profinite}, as appeared below.  
	
	Notice that in case $I$ is finite, the condition of being converging to 1 always holds. Moreover, one easily verifies that for every set $I$ of pro-$\mathcal{C}$ groups, the free pro-$\mathcal{C}$ product exists and is unique up to isomorphism. In addition, the following holds:
	\begin{prop}\cite[Proposition 5.1.6]{ribes2017profinite}
		Let $\{G_i\}_{i\in I}$ be a set of pro-$\mathbb{C}$ groups, and $\coprod_I^{\mathcal{C}}G_i$ their free pro-$\mathcal{C}$ product. Let $\{\varphi_i:G_i\to G\}_{i\in I}$ be the set of natural homomorphisms. Then for all $i\in I$, $\varphi_i:G_i\to G$ is a monomorphism.
		
	\end{prop}
	
	In \cite{jarden1996infinite} Moshe Jarden suggested the following question: Let $n$ be a finite number. Is the free product of $n$ absolute Galois groups an absolute Galois group as well?
	
	This question was answered in the affirmative in several papers such as \cite{ershov1997free, haran2000free, ershov2001free} and \cite{Koenigsmann2005Products}. In particular, if all the groups can be realized as absolute Galois groups over fields of common characteristic $l$ ($l\geq 0$) then so can their free profinite product. In case the groups $G_1,...,G_n$ are pro-$p$ groups it can be shown that their free \textbf{pro-$p$} product is an absolute Galois group as well (over the same common characteristic). (\cite[Remark 3.5]{haran2000free}).
	
	It is worth mentioning the following result of Koenigsmann \cite{Koenigsmann2005Products}:
	\begin{thm}
		Let $\mathcal{C}$ be a class of finite groups closed under forming subgroups, quotients, and extensions, and assume that the class of pro-$\mathcal{C}$ absolute Galois groups is closed under free pro-$\mathcal{C}$ products of a finite number of groups. Then $\mathcal{C}$ is either the class of all finite groups or the class of all finite p-groups.
	\end{thm}
	 Contrary to the finite case, the free profinite product over a set converging to 1 of infinitely many absolute Galois groups may not be an absolute Galois group of any field, as can be seen later in Lemma \ref{main example}, as well as an example in \cite{mel1999free}. 
	
		In his paper from 1999 \cite{mel1999free} Mel'nikov gave the following criterion for a free product over a separable sheaf to be realized as an absolute Galois group:
	\begin{thm}
		Let $T$ be a separable profinite space and $\mathcal{G}=\bigcup_T G_t$ a sheaf of profinite groups defined over $T$ and let $l\geq 0$. Then $\prod_TG_t$ is an absolute Galois group of a field of characteristic $l$ if and only if there exists a morphism $\varphi:\mathcal{G}\to G_F$ for some field $F$ of characteristic $l$, such that $G_F$ is separable and for every $t\in T$ the restriction of $\varphi$ to $G_t$ is injective.
	\end{thm}
Here and below $G_F$ ($G_F(p)$) stands for the absolute (maximal pro-$p$) Galois group of the field $F$.
\begin{rem}
Mel'nikov's Criterion can be extended to every sheaf of profinite groups and every field $F$ by the exact same proof. More precisely, Mel'nikov used the given morphism  $\varphi:\mathcal{G}\to G_F$ in order to construct an embedding $ \prod_TG_t\to G_F\coprod F[T]$ where $F[T]$ denotes the free profinite group over the profinite space $T$. The separability condition of $G_F$ and $T$ came in order to conclude that $G_F\coprod F[T]$ is an absolute Galois group, as the closeness under free profinite product was only known in that days for the class of \textit{separable} absolute Galois groups. Now that Jarden's question has been fully solved in the affirmative, the separability condition can be removed.
\end{rem}
In this paper we only need the original separable version of Mel'nikov's Criterion, as we are going to deal with countable set of separable groups.

	Notice that the free profinite product over the set $I$ converging to 1 is in fact the free product over $T=I\cup\{\ast\}$, the one-point compactification of the discrete space $I$, where $G_{\ast}=\{e\}$. Moreover, a morphism of sheaves into a pro-$\mathcal{C}$ group in that case is nothing but a converging to 1 set of homomorphisms. Hence, Mel'nikov Criterion can be phrased for a free product over a set converging to 1 as follows:
	
	Let $\{G_i\}_{i\in I}$ be a set of profinite groups. Then the free profinite product $\coprod_IG_i$ over the set $I$ converging to 1 can be realized as an absolute Galois group over a field of characteristic $l$ if and only if there is a converging to 1 set of homomorphisms $\{G_i\}\to G_F$ where $F$ is a field of characteristic $l$.
	
		From now on, unless stated otherwise, when we talk about a free pro-$\mathcal{C}$ product over a set $I$, we will always refer to the free pro-$\mathcal{C}$ product over the set $I$ converging to 1. 
	
	We also recall the following useful fact:
	\begin{fact}\cite[Proposition 1.3(a)]{mel1999free}
Let $l$ be a prime and $G$ a profinite group which occurs as an absolute Galois group over a field of characteristic $l$. Then $G$ can occur as an absolute Galois group over a field of characteristic 0.
	\end{fact}

	We use Mel'nikov's Criterion in order to give a simple Criterion for a countable series of Demushkin groups of rank $\aleph_0$ to create an absolute Galois free product. More precisely, we prove the following:
	
	\begin{thm*}
	Let $p$ be a fixed prime, $l$ a prime different then $p$, $I$ be an infinite countable set and $\{G_i\}_{i\in I}$ be a set of pro-$p$ Demushkin groups of rank $\aleph_0$ which can be realized as absolute Galois groups. Then the following are equivalent:
\begin{enumerate}
	\item $\coprod_{i\in I}G_i$ can be realized as an absolute Galois group of a field $F$ (of characteristic $l$).
	\item $\coprod_{i\in I}^pG_i$ can be realized as an absolute Galois group of a field $F$ (of characteristic $l$).
	\item $\coprod_{i\in I}G_i$ can be realized as an absolute Galois group of a field $F$ (of characteristic $l$) which contains a primitive $p$'th root of unity.
	\item $\coprod_{i\in I}^pG_i$ can be realized as a maximal pro-$p$ Galois group of a field $F$ (of characteristic $l$) which contains a primitive $p$'th root of unity.
	\item $\log _pq(G_i)\to \infty$ (and $\log_pq(G_i)\geq f(l,p)$ for all $i$ such that $q(G_i)\ne 2$. In addition if $q(G_i)=2$ then $G_i$ can be realized over a field of characteristic $l$).
\end{enumerate}  
	\end{thm*}
The value $f(l,p)$ will be presented later.

	The following natural follow-up questions remain open:
	\begin{question}\label{further questions}
		\begin{itemize}
			\item Is the free pro-$p$ product of countably many finitely generated Demushkin groups, which can occur as maximal pro-$p$ Galois groups of local fields, and satisfy  $\log_pq(G_i)\to \infty$, a maximal pro-$p$ Galois group as well?
			\item Can we generalize Theorem \ref{free product of Demushkin} to uncountable sets of Demushkin groups of rank $\aleph_0$? I.e, assuming that $I$ is an uncountable set of pro-$p$ Demushkin groups of rank $\aleph_0$ satisfying that for every natural $k$, there are only finitely many Demushkin groups $G_i\in I$ such that $q(G_i)=p^k$, is $\coprod IG_i$ ($\coprod_I^pG_i$) an absolute Galois group?
			\item What can be said about a general free product over a profinite space of a sheaf consisting of Demushkin groups of rank $\aleph_0$? 
		\end{itemize}
	\end{question}

	\section*{Main results}
	We start this section with a minor improvement to the closeness of the class of absolute Galois group under finite free profinite product:
	\begin{rem}\label{common subfield}
		Let $L$ be a field and let $K_1,K_2$ be fields which contain isomorphic copies of $L$, which we refer to as $L_1,L_2$ correspondingly. Then $G_{K_1}\coprod G_{K_2}$ can be realized as an absolute Galois group over a field containing an isomorphic copy of $L$.  
	\end{rem} 
\begin{proof}
First we realize both $G_{K_1},G_{K_2}$ as absolute Galois groups of separable algebraic extensions $F_1,F_2$ of the same field $F$ which contains a copy of $L$. We do so by replacing the role of $F_0$ and $E_0$ in the proof of \cite[Proposition 2.5]{haran2000free} by $L(L_i)$. More precisely, choose transcendence bases $T_1$ and $T_2$ for $K_1/L_1$ and $K_2/L_2$ correspondingly and put $M_1=L_1[T_1],M_2=L_2[T_2]$. Let $\varphi_i:L\to L_1\cup \{\infty\}$, $\varphi_2:L\to L_2\cup\{\infty\}$ be the places defined by the given isomorphisms. Choose a set $T$ of cardinality greater then $\max\{|T_1|,|T_2|\}$ and surjective maps $\varphi'_i:T\to T_i$ for $i=1,2$. Put $F=L[T]$. Then $\varphi_i,\varphi'_i$ can be extended to places $\varphi_i'':F\to M_i\cup\{\infty\}$. Denote the corresponding valuations by $v_i$, \cite[Corollary 2.3(b)]{haran2000free} gives fields $F_i$ which are algebraic over $F$ such that $G_{F_i}\cong G_{K_i}$. Replace $F,F_1,F_2$ by their separable closures we get the desired fields. Now apply \cite[Theorem 3.3]{haran2000free} to realize $G_{K_1}\coprod G_{K_2}$ over some extension of $F$, we are done.
\end{proof}
\begin{cor}\label{common root of unity}
	Let $l\geq 0$ and $p$ be prime, and let $G_1,G_2$ be profinite groups that can be realized as absolute Galois groups over fields $K_1,K_2$ of characteristic $l$ that contain a primitive $p$-th root of unity. Denote by $L'$ the prime field of $K_1,K_2$ and let $L=L'[\rho]$ for $\rho$ a primitive $p$-th root of unity. Then by Remark \ref{common subfield} we conclude that $G_1\coprod G_2$ can be realized over a field of characteristic $l$ that contains a primitive $l$-th root of unity.
\end{cor}
Using the above Corollary, we can now suggest a similar improvement to Mel'nikov's Criterion.
\begin{lem}\label{common root infinite case}
	Let $T$ be a profinite space, $\mathcal{G}=\bigcup_T G_t$ a sheaf of profinite groups defined over $T$. In addition,  let $l\geq 0$ and $p$ be a prime. Then $\prod_TG_t$ is an absolute Galois group of a field of characteristic $l$ \textbf{that contains a primitive $p$-th root of unity} if and only if there exists a morphism $\varphi:\mathcal{G}\to G_F$ for some field $F$ of characteristic $l$ \textit{that contains a primitive $p$-th root of unity}, such that for every $t\in T$ the restriction of $\varphi$ to $G_t$ is injective.
\end{lem}
\begin{proof}
	Mel'nikov's proof uses the morphism $\varphi$ in order to construct an embedding of $\prod_TG_t$ into $G_F\coprod F[T]$ where $F[T]$ is the free profinite group over the profinite space $T$. By Corollary \ref{common root of unity} it is enough to show that $F[T]$ can be realized over a field of characteristic $l$ that contains a primitive $p$-th root of unity. This holds by \cite[Example 3.3.8 (e)]{ribes2000profinite} which realizes every free profinite group over a set $X$ converging to 1 as the absolute Galois group of $K(t)$ for every algebraically closed field $F$ of cardinality $|X|$, and the fact that every free profinite group over a profinite space is in fact a free profinite group over some set $X$ converging to 1, where $|X|=\omega_0(F[T])$. (\cite[Proposition 3.5.12]{ribes2000profinite}). Now let $K=\mathbb{F}_l[\rho_p]$ for $l\ne 0$ and $K=\mathbb{Q}[\rho_p]$ for $l=0$, and let $K'$ be the algebraic closure of $K(x_i)_{i\in I}$ for $|I|=\omega_0(F[T])$, then $G_{K'(t)}$ is isomorphic to $F[T]$.
\end{proof}

	Now we discuss the connection between the realization as an absolute Galois group of the free profinite and free pro-$p$ products of a set of pro-$p$ groups. First we need the following lemma:
\begin{lem} \label{the pro-p case}
	Let $\{G_i\}_{i\in I}$ be a set of pro-$p$ groups. Then there is a natural embedding $\coprod_I^pG_i\to \coprod_IG_i$.
\end{lem}
\begin{proof}
	We use a similar proof to that of \cite[Remark 3.5]{haran2000free} in order to construct an embedding $\coprod_{I}^pG_i\to  G=\coprod_{I}G_i$. Denote by $\alpha: G=\coprod_{I}G_i\to \coprod_{I}^pG_i$ the homomorphism defined by the converging to 1 set of natural homomorphisms $\alpha_i:G_i\to \coprod_{I}^pG_i$. Since $ \coprod_{I}^pG_i$ is generated by $\bigcup_i\alpha_i(G_i)$, $\alpha$ is in fact an epimorphism. Choose a $p$-Sylow subgroup $P$ of  $\coprod_{I}G_i$. Then $\alpha(P)$ is a $p$-Sylow subgroup of $\coprod_{I}^pG_i$ and hence $\alpha(P)=\coprod_{I}^pG_i$. We identify each $G_i$ with its image in $\coprod_{I}G_i$. Thus, each $G_i$ is a subgroup of $\coprod_{I}G_i$. Since $G_i$ is a pro-$p$ group, there exists some $a_i\in \coprod_{I}G_i$ such that $G_i^{a_i}\leq P$. Choose $b_i\in P$ which satisfies $\alpha(b_i)=\alpha(a_i)$. Then $G_i^{a_ib_i^{-1}}\leq P$. The homomorphism $g\to g^{a_ib_i^{-1}}$ defines an embedding of $G_i$ into $P$. In fact, we get a converging to 1 set homomorphism $\{G_i\to P\}$. Indeed, let $U\leq_o P$. We may assume that $U$ is normal. Thus there exists $V\unlhd _o G$ such that $V\cap P\subseteq U$. There exists a finite subset $J\subseteq I$ such that for every $i\in I\setminus J$, $G_i\subseteq V$. Since $V$ is normal  we conclude that for every $i\in I\setminus J$ $G_i^{a_ib_i^{-1}}\leq V$. Hence  for every $i\in I\setminus J$ $G_i^{a_ib_i^{-1}}\leq V\cap P\leq U$. Denote the homomorphism $\coprod_I^pG_i\to P$ induced by this set by $\alpha'$. Such a homomorphism exists by definition of the free pro-$p$ product, since $P$ is a pro-$p$ group. Eventually, $\alpha(\alpha'(g))=g^{\alpha(a_i)\alpha(b_i)^{-1}}=g$ for each $i$ and each $g\in G_i$. Hence, $\alpha\circ \alpha':\coprod_I^pG_i\to \coprod_I^pG_i$ is an isomorphism. Thus, $\alpha': \coprod_{I}^pG_i\to  G=\coprod_{I}G_i$ is an embedding, as required. 
\end{proof}
Now we move to talk about Demushkin groups. We start with some general information.

Recall that a pro-$p$ Demushkin group is a pro-$p$ group $G$ which satisfies 
\begin{enumerate}
	\item $\dim H^2(G)=1$.
	\item The cup product bilinear form $H^1(G)\cup H^1(G)\to H^2(G)\cong \F_p$ is nondegenerate.
\end{enumerate}
where $H^i(G):=H^i(G,\F_p)$.

The theory of finitely generated Demushkin group was studied extensively in \cite{serre1962structure,demushkin1961group,demushkin19632}, and \cite{labute1967classification}. This theory was first extended to Demushkin group of rank $\aleph_0$ in 1966 by Labute (\cite{labute1966demuvskin}). Demushkin groups of countably rank come equipped with 3 invariants that in most cases determine the group up to isomorphism:

$\mathbf{q(G)}$:

Since $\dim(H^2(G))$ equals the minimal number of relations required to define $G$ (see, for example, \cite[Section 1.2]{labute1966demuvskin}), a Demushkin group is always 1-related. Hence $G$ has the form $F/r$ where $F$ is a free pro-$p$ group of the same rank, and $r\in \Phi(F)$. Here $\Phi(F)=F^p[F,F]$ is the Frattini subgroup of $F$. We define $q(G)=p^n$ for the maximal $n\in \mathbb{N}\cup \{\infty\}$ such that $r\in F^{p^n}[F,F]$. We consider $p^{\infty}$ to be 0- hence we say that $q(G)=0$ if and only if $r\in [F,F]$. This notation makes sense since for every element $x$ in a pro-$p$ group, $x^{p^{\infty}}=e$. In particular, we define $\log_p0$ to be $\infty$.  Observe that this definition is well-defined, independent of the choice of a basis of $F$. In fact, $q(G)$ can also be characterized as follows: since $G$ is 1-related, $G/[G,G]\cong \Z/p^n\times \Z_p^{\rank G-1}$, for   some $n\in \mathbb{N}\cup \{\infty\}$. This $p^n$ equals to $q(G)$. 

In the papers \cite{minavc1991demuvskin,minavc1992pro} it has been shown that if $K$ is a field of characteristic different than $p$ which contains a primitive $p$-th root of unity, such that $G_K(p)$ is a pro-$p$ Demushkin group of rank $\aleph_0$, then $\log_pq(G)$ equals the maximal natural number $n$ such that $K$ contains a primitive $p^n$-th root of unity. If $\log_pq(G)=\infty$ then $K$ contains all $p^n$-th roots of unity for every natural number $n$. This result follows from the connection between the invariants $q(G)$ and $\Img(\chi)$, as explained below.

{$\mathbf{s(G):}$}

By \cite{labute1966demuvskin}, every countably generated Demushkin group has finite cohomological dimension 2, and hence admits a dualizing module $I$. Moreover, it was shown that $I\cong \Q_p/\Z_p\lor \Z/q$ where $q$ is some natural power of $p$. We denote $s(G)=0$ in the first case and $s(G)=q$ in the later case. In \cite[Theorem 2.2]{minavc1992pro} it has been proven that if $K$ is a field of characteristic different than $p$ such that $G_K(p)$ is a pro-$p$ Demushkin group of rank $\aleph_0$, then $s(G)=0$.

{$\boldsymbol{\Img(\chi)}$:}

The dualizing module $I$ comes equipped with a homomoprhism $\chi:G\to \operatorname{Syl}(\operatorname{Aut}(I))\cong \operatorname{Syl}((\Z_p/s(G))^{\times})$, which is called \textit{the character}. 
\begin{rem}\label{interpetation for img chi}
	In \cite{labute1966demuvskin} it has been proven that similarly to the finitely generated case, for $q(G)\ne 2$, $\Img(\chi)=1+q(G)\Z_p/s(G)$. For $q(G)=2$ $\Img(\chi)$ comes from some list of subgroups of $1+2\Z_p$, none of them is contained in $1+4\Z_2$.
\end{rem}
	Using the above interpretation of $q(G)$, we can explain the Galois interpretation of $q(G)$. By \cite{minavc1991demuvskin,minavc1992pro}, if $K$ is a field containing a primitive root of unity of order $p$ such that $G_K(p)$ is a Demushkin group $G$ of rank $\aleph_0$, then the dualizing module $I$ is isomorphic to $\bigcup_n\mu_{p^n}$, the set of all $p$-powers roots of unity, while the character $\chi:G_K(p)\to \operatorname{Aut}(\bigcup_n\mu_{p^n})\cong (\Z_p)^{\times}$ is the natural homomorphism, which hence can be written explicitly as the map $\chi:G\to (\Z_p)^{\times}$ which assigns to each $\chi(\sigma)$  the unique $p$-adic number such that $\sigma(\rho)=\rho^\chi(\sigma)$ for every $\rho\in \bigcup_n\mu_{p^n}$. Now let $n$ be natural and $\rho_{p^n}$ a $p^n$-th root of unity. Then $\rho_{p^n}\in K$ if and only if it is preserved by the action of $G_K(p)$, if and only if $\rho_{p^n}^{\chi(\sigma)}=\rho_{p^n}$ for every $\sigma\in G$ if and only if $\operatorname{Im}(\chi)\leq 1+p^n\Z_p$.

Eventually, Min{\'a}{\v{c}}\& Ware proved the following characterization of Demushkin groups of rank $\aleph_0$ which can be realized as absolute Galois groups:
\begin{prop}\cite[Theorems 1.2+Theorem 3.1+Theorem 2.3'']{minavc1991demuvskin} \label{characterisation of Demushkin groups as absolute Galois groups}
	Let $p$ be a prime, $G$ be a pro-$p$ Demushkin group of rank $\aleph_0$ with $q(G)\ne 2$ and $l$ a prime different than $p$. Then $G$ can be realized as an absolute Galois group (over a field of characteristic $l$) if and only if $G$ can be realized as a maximal pro-$p$ Galois group (over a field of characteristic $l$ that contains a primitive $p$-th root unity) if and only if $s(G)= 0$ (and $\log_pq(G)\geq f(l,p)$). Here $f(l,p)$ is some natural number associated to $p$ and $l$ (see \cite[Theorem 3.1]{minavc1991demuvskin}, \cite[Theorem 6.1]{minavc1992pro}).    
\end{prop}
\begin{rem}
	\begin{enumerate}
		\item If the absolute Galois group of a field $K$ is a pro-$p$ group then $K$ must contain a primitive $p$-th root of unity $\mu_p$, for otherwise the normal extension $K[\mu_p]/K$ is not a $p$-extension.
		\item The case of characteristic $p$ can be ignored since the maximal pro-$p$ Galois group of a field of characteristic $p$ is always a free pro-$p$ group (\cite[Corollary 1,II-5]{serre1994cohomologie}).
		\item The case of $p=2$ has been dealt with in \cite[Theorem 3.2'']{minavc1991demuvskin}, where a full classification of pro-2 Demushkin groups of rank $\aleph_0$ which can be realized as absolute or maximal pro-2 Galois groups (over a given characteristic $l$) has been given, in terms of the invariants $t(G),  \Img(\chi)$. Since for $q(G)\ne 2$ we must have $t(G)=1$ and $\Img(\chi)=U^(f)_2$ (a classification of pro-$p$ Demushkin groups of rank $\aleph_0$ in terms of their invariants is given in \cite{labute1966demuvskin}) then we get the above criteria.
	\end{enumerate}
\end{rem}
Before we can prove the restriction on a set of pro-$p$ Demushkin groups whose free profinite (pro-$p$) product can be realized as a maximal pro-$p$ Galois group, we need one more remark. We call a set $\{H_i\}_{i\in I}$ of subgroups of $G$ a \textit{converging to 1 set of subgroups} if for every $U\leq_o G$, $H_i\leq U$ for almost all $i\in I$.
	\begin{rem}\label{converging to 1 extensions}

	Let $K$ be a field and $\{K_i/K\}_{i\in I}$ a set of separable  pro-$\mathcal{C}$ field extensions. I.e, separable  field extensions whose Galois group $\operatorname{Gal}(K_i/K)$ is a pro-$\mathcal{C}$ group. We say that $\{K_i/K\}_{i\in I}$ is a converging to 1 set of field extensions in $(K)^{\mathcal{C}}$ if every $x\in {K}^{\mathcal{C}}$ is contained in all but finitely many $K_i$'s, where ${K}^{\mathcal{C}}$ denotes the maximal separable  pro-$\mathcal{C}$ extension of $K$. In case $\mathcal{C}$ is the variety of all finite groups, ${K}^{\mathcal{C}}$ is nothing but the separable closure of $K$, and we denote it by $K^{\operatorname{sep}}$. In addition, the maximal pro-$p$ extension of $K$ is usually denoted by $K(p)$. 
	
	One can easily verify that $\{K_i/K\}_{i\in I}$ is a converging to 1 set of field extensions in  ${K}^{\mathcal{C}}$ if and only if $\{\operatorname{Gal}({K_i}^{\mathcal{C}}/K_i)\}_{i\in I}$ is a converging to 1 set of subgroups of $\operatorname{Gal}( {K}^{\mathcal{C}}/K)$. Indeed, Let $U\leq_o \operatorname{Gal}( {K}^{\mathcal{C}}/K)$ and assume that $\{K_i/K\}_{i\in I}$ is a converging to 1 set of field extensions in  ${K}^{\mathcal{C}}$. Let $({K}^{\mathcal{C}})^U$ be the fixed field of $U$. Since $U$ is open, $[({K}^{\mathcal{C}})^U:K]<\infty$. Choose a basis $x_1,...,x_n$ of $ ({K}^{\mathcal{C}})^U$ over $K$. For every $1\leq t\leq n$ there is a finite subset $J_t$ of $I$ such that $x_j\in K_i$ for all $i\in I\setminus J_t$. Set $J=\bigcup_{t=1}^nJ_t$, then $({K}^{\mathcal{C}})^U\subseteq K_i$ for all $i\in I\setminus J$. Taking the stabilizer of each subfield in the action of $\operatorname{Gal}( {K}^{\mathcal{C}}/K)$ over ${K}^{\mathcal{C}}$, the inclusion reversed, so we are done. The second direction is proved in a similar way.
\end{rem}

\begin{lem}\label{main example}
 Let $p$ be a prime and let $\{G_i\}_{i\in I}$ be a set of pro-$p$ Demushkin groups of rank $\aleph_0$. If $\coprod_I^pG_i$ is a maximal pro-$p$ Galois group of a field of characteristic different then $p$ which contains a primitive $p$-th root of unity, then for every natural number $n$ there are only finitely many Demushkin groups $G_i$ for which $\log_pq(G_i)\leq n$.
\end{lem}
\begin{proof}
  Assume that there exists a field $F$ containing e primitive $p$-th  root of unity such that $G_F\cong \coprod_I^p G_i$. Identify each $G_i$ with its natural image in $\coprod_I^p G_i$, then $G_i\ne G_j$ for all $i\ne j$ (see \cite[Proposition 5.1.6]{ribes2017profinite}). In addition, by definition of a free pro-$\mathcal{C}$ product, $\{G_i\}_{i\in I}$ is a converging to 1 set of subgroups of $G_F$. By Remark \ref{converging to 1 extensions}, $\{\bar{F}^{G_i}\}_{i\in I}$ is a converging to 1 set of field extensions of $F$ inside $F(p)$. Denote $F_i=F(p)^{G_i}$. Now let $n$ be some natural number. Since $\rho_p\in F$, $\rho_{p^n}\in F(p)$. Here $\rho_{p^n}$ denotes a primitive $p^n$-th root of unity. Then there is a finite subset of $J$ such that for all $i\in I\setminus J$, $\rho_{p^n}\in F_i$. Recall that $G_i=G_{F(p)^{G_i}}(p)$. By the Galois interpretation of $q(G)$ that was described above we get that for every $i\in I\setminus J$, $\log_pq(G_i)\geq n$. 
\end{proof}

In order to prove the main theorem we need a few more lemmas.
\begin{lem}\cite[Remark after Theorem 3.4.1]{neukirch2013cohomology}\label{open subgroup}
	For every profinite group of finite cohomological dimension $G$, and an open subgroup $U\leq G$ the dualizing module of $U$ equals the dualizing module of $G$ with the induced action.
\end{lem}
\begin{lem}\label{all with q=0}
	Let $p$ be a fixed prime. Let $\{G_n\}_{n\in \mathbb{N}}$ be a set of pro-$p$ Demushkin groups of rank $\aleph_0$, satisfying $s(G_n)=q(G_n)=0$ for all $n$. Then there is a pro-$p$ Demushkin group $G$ of rank $\aleph_0$ and $q(G)=s(G)=0$ equipped with a converging to 1 set of monomorphisms $\{\varphi_n:G_n\to G\}_{n\in \mathbb{N}}$.
\end{lem}
\begin{proof}
	Recall that by \cite[Corollary 1]{labute1966demuvskin}, all pro-$p$ Demushkin groups of rank $\aleph_0$ with $q(G)=s(G)=0$ are isomorphic. Let $G$ be a pro-$p$ Demushkin group of rank $\aleph_0$ with $q(G)=s(G)=0$. Since $G$ has countable rank, by \cite[Propositions 2.6.1+2.6.2]{ribes2000profinite}, $G$ admits $\aleph_0$ open subgroups. Let $\{U_n\}_{n\in \mathbb{N}}$ be an indexing of the set of all open subgroups of $G$. Define $V_n=\bigcap_{m=1}^n U_m$. By \cite[Theorem 2]{labute1966demuvskin} $V_n$ is a pro-$p$ Demushkin group too. Moreover, $U$ has rank $\aleph_0$ as an open subgroup of a group of rank $\aleph_0$. In addition, by Lemma \ref{open subgroup}, $s(V_n)=0$. Eventually, since $\Img(\chi)=1+0\Z_p=\{e\}$, $\Img(\chi_{V_n})=\{e\}$ which implies $q(V_n)=0$. Now let $\varphi_n:G_n\to G$ be the composition of an isomorphism $G_n\to V_n$ with the inclusion map. Then  $\{\varphi_n:G_n\to G\}_{n\in \mathbb{N}}$ is a converging to 1 set of monomorphisms.
\end{proof}
\begin{lem}\label{all with $q$ natural goes to infinity}
	Let $p$ be a fixed prime. Let $\{G_n\}_{n\in \mathbb{N}}$ be a set of pro-$p$ Demushkin groups of rank $\aleph_0$, satisfying $s(G_n)=0$ for all $n$ and $q(G_n)=q_n$ where $q_n>2$ is a series of natural $p$-powers converging to infinity. Then there is a pro-$p$ Demushkin group $G$ of rank $\aleph_0$ and $s(G)=0, q(G)\ne 2$ equipped with a converging to 1 set of monomorphisms $\{\varphi_n:G_n\to G\}_{n\in \mathbb{N}}$.
\end{lem}
\begin{proof}
	Let $G$ be a pro-$p$ Demushkin group with $s(G)=0$ and $q(G)=q$ for $q=\min \{q(G_i)\}_{i\in I}$. As we stated in Lemma \ref{all with q=0}, every open subgroup $U$ of $G$ is a Demushkin group of $s(U)=0$. As in the proof of Lemma \ref{all with q=0}, we index the open subgroups $U$ of $G$ by $\mathbb{N}$ and let $V_n=\bigcap_{m=1}^nU_m$. In particular, every $V_n$ is a Demushkin group with $s(V_n)=0$. Recall that $q(V_n)=q$ for $\chi(V_n)=1+q\Z_p$. Since $V_n$ is of finite index in $G$, $\chi$ has finite index in $\chi(G)$. Hence, for every $n$, $q(V_n)$ is finite. Moreover,  $q(V_n)\to \infty$. Indeed, for every $q'>q$, let $U_m=\chi^{-1}(1+q'\Z_p)$, then $q(V_m)\geq q(U_m)=q'$. Recall again that for every $q\ne 2$ a power of $p$ there is a unique pro-$p$ Demishkin group $H$ with $s(H)=0$ and $q(H)=q$ up to isomorphism. Now define the following monomorphisms $\varphi_n:G_n\to G$ as follows: for every $n$, let $m_n$ be the greatest integer such that $q(V_{m_n})\leq q(G_n)$ and send $V_n$ isomorphically onto some open subgroup $U'$ of $V_{m_n}$ with $q(U')=q(G_n)$; such an open subgroup can be constructed, for example, by taking $V_{m_n}\cap\chi^{-1}(1+q(G_n)\Z_p)$. Since $q(V_n)\to \infty$,  $\{\varphi_n:G_n\to G\}_{n\in \mathbb{N}}$ is a converging to 1 set of monomorphisms.
\end{proof}
\begin{lem}\label{splitting the free product}
	Let $I$ be an indexing set and assume that $I=I_1\cup\cdots \cup I_n$. Let $\{G_i\}_{i\in I}$ be a set of pro-$\mathcal{C}$ groups. Then $\coprod_I^{\mathcal{C}}G_i\cong \coprod\limits_{i=1}^{n}{}{\!\!^\mathcal{C}}(\coprod_{i\in I_n}^{\mathcal{C}}G_i)
	$
\end{lem}
\begin{proof}
	We shall show that $\coprod\limits_{i=1}^{n}{}{\!\!^\mathcal{C}}(\coprod_{i\in I_n}^{\mathcal{C}}G_i)$ satisfies the universal property of free pro-$\mathcal{C}$ product. Let $\{\varphi_i:G_i\to H\}_{i\in I}$ be a converging to 1 set of homomorphisms into a group in $\mathcal{C}$. Then for every $1\leq m\leq n$, $ \{\varphi_i:G_i\to H\}_{i\in I_m}$ is a converging to 1 set of homomorphisms. Denote by $\gamma_i:G_i\to \coprod_{i\in I_m}^{\mathcal{C}}G_i$ the natural homomorphism for $i\in I_m$. Then for every $m$ there exists a homomorphism $\psi:\coprod_{i\in I_m^{\mathcal{C}}}G_i\to H$ such that $\psi_m\circ \gamma_i=\varphi_i$ for every $i\in I_m$. Now let $\delta_m:\coprod_{i\in I_m}^{\mathcal{C}}G_i\to  \coprod\limits_{i=1}^{n}{}{\!\!^\mathcal{C}}(\coprod_{i\in I_n}^{\mathcal{C}}G_i)$ be the natural homomorphism. Again, by definition of free pro-$\mathcal{C}$ product there exists $\eta: \coprod\limits_{i=1}^{n}{}{\!\!^\mathcal{C}}(\coprod_{i\in I_n^{\mathcal{C}}}G_i)\to H$ such that $\eta\circ \delta_m=\psi_m$. Taking $f_i:G_i\to  \coprod\limits_{i=1}^{n}{}{\!\!^\mathcal{C}}(\coprod_{i\in I_n}^{\mathcal{C}}G_i)$ to be $\delta_m\circ \gamma_i$ whenever $i\in I_m$, we get the required. 
\end{proof}
Now we are ready to prove the main theorem of the paper.
\begin{thm}\label{free product of Demushkin}
	Let $p$ be a fixed prime $l$ a prime different than $p$, $I$ be an infinite countable set and $\{G_i\}_{i\in I}$ pro-$p$ Demushkin groups of rank $\aleph_0$ which can be realized as absolute Galois groups. Then the following are equivalent:
	\begin{enumerate}
		\item $\coprod_{i\in I}G_i$ can be realized as an absolute Galois group of a field $F$ (of characteristic $l$).
		\item $\coprod_{i\in I}^pG_i$ can be realized as an absolute Galois group of a field $F$ (of characteristic $l$).
		\item $\coprod_{i\in I}G_i$ can be realized as an absolute Galois group of a field $F$ (of characteristic $l$) which contains a primitive $p$'th root of unity.
		\item $\coprod_{i\in I}^pG_i$ can be realized as a maximal pro-$p$ Galois group of a field $F$ (of characteristic $l$) which contains a primitive $p$'th root of unity.
		\item $\log _pq(G_i)\to \infty$ (and $\log_pq(G_i)\geq f(l,p)$ for all $i$ such that $q(G_i)\ne 2$. In addition if $q(G_i)=2$ then $G_i$ can be realized over a field of characteristic $l$).
	\end{enumerate} 
\end{thm}
\begin{proof}
	By Lemma \ref{the pro-p case} we get that  $(1)\Rightarrow(2)$. Assume $(2)$, then $\coprod_{i\in I}^pG_i$ is isomorphic to the absolute Galois group of a field $F$. Since $\coprod_{i\in I}^pG_i$ is a pro-$p$ group, $F$ must contain a primitive $p$-th root of unity. Since there is a converging to 1 set of monomorphisms $\{G_i\to \coprod_{i\in I}^pG_i\}$ then by Lemma \ref{common root infinite case} $\coprod_{i\in I}G_i$ can be realized as an absolute Galois group over e field containing a primitive $p$-th root of unity. We get that $(2)\Rightarrow (3)$. $(3)\Rightarrow(4)$ is immediate since the maximal pro-$p$ Galois group of a field is the maximal pro-$p$ quotient of its absolute Galois group, and the maximal pro-$p$ quotient of $\coprod_{i\in I}G_i$ is $\coprod_{i\in I}^pG_i$ (see \cite[Theroem 5.6.1]{ribes2017profinite}). $(4)\Rightarrow (5)$ is precisely Lemma \ref{main example}, (together with the fact that if $\coprod_I^pG_i$ can be realized as a maximal pro-$p$ Galois group over $F$, then $G_i$, being a closed subgroup of $\coprod_{I}^pG_i$ can be realized as a maximal pro-$p$ Galois group of some extension of $F$).  We left to prove $(5)\Rightarrow (1)$. Let $n$ be the number of Demushkin groups $G_i,i\in I$ for which $q(G_i)=2$. For more convenience denote these groups by $G_1,...,G_n$. Observe that since $q(G_i)\to \infty$ then $n$ must be finite. In addition let $I_1$ be the subset of all $G_i,i\in I$ with $q(G_i)=0$ and $I_2$ be the subset of all $G_i,i\in I$ with $q(G_i)\ne 0,2$.  Lemmas \ref{all with q=0} and \ref{all with $q$ natural goes to infinity} give converging to 1 sets of monomorphisms $\{G_i\to H_1\}_{i\in I_1}$ and $  \{G_i\to H_2\}_{i\in I_2}$ where $H_1,H_2$ are pro-$p$ Demushkin groups with $s(H_i)=0$ and $q(H_i)\ne 2$ (in addition, $\log_pq(H_i)\geq f(l,p)$). By Proposition \ref{characterisation of Demushkin groups as absolute Galois groups}, $H_1, H_2$ can be realized as absolute Galois groups (over fields of characteristic $l$). Hence, Mel'nikov's Criterion implies that  $\coprod_{I_1}G_i$, $\coprod_{I_2}G_i$ can be realized as absolute Galois groups (over fields of characteristic $l$). Now applying the closeness of the class of absolute Galois groups under free profinite product of finitely many groups, we conclude that $G_1\coprod\cdots \coprod G_n\coprod\left(\coprod_{I_1}G_i\right)\coprod\left(\coprod_{I_2}G_i\right)$ can be realized as an absolute Galois group (over a field of characteristic $l$). By Lemma \ref{splitting the free product} we are done.
\end{proof}
\begin{rem}
	The case of characteristic $p$ can be ignored as the maximal pro-$p$ Galois group of a field of characteristic $p$ must be free, as stated above.
\end{rem}
\begin{rem}
	We can now simplify Question \ref{further questions} (2) and focus only on the case that $q(G_i)=0$ for all $G_i\in I$. Indeed, let $I$ be an uncountable set of pro-$p$ Demushkin groups satisfying the conditions of Question \ref{further questions} (2). Define $I_1=\{G_i:q(G_i)\ne 0\}$ and $I_2=\{G_i:q(G_i)= 0\}$. By assumption $I_1$ is countable, and satisfying that $\log_pq(G_i)\to\infty$. Thus by Theorem \ref{free product of Demushkin} $\coprod_{I_1}G_i$ ($\coprod_{I_1}G_i$) is an absolute Galois group. By Lemma \ref{splitting the free product} $\coprod_IG_i\cong (\coprod_{I_1}G_i)\coprod(\coprod_{I_2}G_i)$ ($\coprod_I^pG_i\cong (\coprod_{I_1}^pG_i)\coprod^p(\coprod_{I_2}^pG_i)$). Hence $\coprod_IG_i$ ($\coprod_I^pG_i$) is an absolute Galois group if and only if so is $\coprod_{I_2}G_i$ ($\coprod_{I_2}^pG_i$).
\end{rem}

	\subsection*{Acknowledgments}
	The author wish to thank Prof. Dan Haran and Moshe Jarden for a useful discussion on infinite free products which cannot be realized as absolute Galois groups.

\end{document}